\documentclass{amsart}
\usepackage{amsrefs}

\newtheorem{theorem}{Theorem}
\newtheorem{proposition}{Proposition}

\newtheorem{remark}{Remark}

\DeclareMathOperator{\inv}{inv}
\DeclareMathOperator{\Eucl}{Eucl}
\DeclareMathOperator{\Jac}{Jac}

\def \rr {\mathbb{R}}

\def \rn {\mathbb{R}^n}

\def \eps {\varepsilon}

\title[Polyharmonic Green functions in domains with small  holes]{Uniform estimates for polyharmonic Green functions in domains with small holes }
\author{Hans-Christoph Grunau}
\email{hans-christoph.grunau@ovgu.de}
\address{Fakult\"at f\"ur Mathematik, Otto-von-Guericke-Universit\"at, Postfach 4120,\newline
39016 Magdeburg, Germany}
\author{Fr\'ed\'eric Robert}
\email{Frederic.Robert@univ-lorraine}
\address{Institut \'{E}lie Cartan,
Universit\'{e} de Lorraine,
B.P. 70239,\newline
54506 Vand{\oe}uvre-l\`{e}s-Nancy Cedex, France}
\date{September 25th 2012.}

\begin{document}
\maketitle
\centerline{\it Dedicated to Patrizia Pucci on the occasion of her 60th birthday.}

\medskip
\noindent 
Given an arbitrary $C^{2k,\theta}$-smooth bounded domain $\Omega\subset \rn$
with exterior unit normal $\nu$, 
$n>2k\geq 2$ and $\theta\in (0,1)$, we define 
$G_\Omega: \overline{\Omega}\times \overline{\Omega}\setminus \{(x,x)/\, x\in \overline{\Omega}\}\to\rr$ 
as the Green function of $(-\Delta)^k$ in the domain $\Omega$ with Dirichlet boundary condition.
This means that for $f\in C^{0,\theta} (\overline{\Omega})$ the unique solution 
$u\in C^{2k,\theta} (\overline{\Omega})$ of the polyharmonic
Dirichlet problem 
\begin{equation}\label{eq:polyharmonic_bvp}
\left\{\begin{array}{ll}
        (-\Delta)^k u=f\qquad &\mbox{\ in\ }\Omega,\\
        u=\partial_\nu u=\ldots =\partial_\nu^{(k-1)} u=0\qquad &\mbox{\ on\ }\partial\Omega
       \end{array}
\right.
\end{equation}
is given by
$$
u(x)=\int_{\Omega}G_\Omega (x,y) f(y)\, dy.
$$
We are interested in pointwise estimates for $G_\Omega$. In the special case $k=1$,
i.e. the case of the usual Laplacian, these can be deduced by using  the maximum
principle. This yields that $G_\Omega$ is positive und bounded from above by the fundamental
solution, i.e. for $n>2$ and any bounded smooth domain $\Omega\subset \rn$ we have
\begin{equation}\label{eq:secondorderestimate}
\forall x,y\in \Omega, x\not=y,:\qquad 0 < G_\Omega (x,y) <\frac{1}{(n-2)ne_n} |x-y|^{2-n}.
\end{equation}
Here, $e_n$ denotes the measure of the $n$-dimensional unit ball. One should observe
that the constant in the right inequality is independent of $\Omega$, even with respect
to singular perturbations.

\medskip\noindent 
When passing to biharmonic or --more general-- polyharmonic equations, i.e. the cases
$k\ge 2$, the maximum principle is no longer available and positivity issues remain
valid only in a very weak and modified sense. Mathematical contributions on this
topic go back at least to Boggio and Hadamard \cites{Boggio,Hadamard}; these
papers are also fundamental for subsequent works on 
estimating polyharmonic Green functions.
For an extensive discussion of related and more recent contributions one may
see the monograph Gazzola-Grunau-Sweers \cite{GGS} and Grunau-Robert \cite{GR}. 
There is no obvious idea how to directly prove higher order analogues
to estimate (\ref{eq:secondorderestimate}). 
However, basing on the general Schauder and $L^p$-theory 
developed by Agmon, Douglis, and Nirenberg \cite{ADN},
Krasovski{\u\i} \cites{K1,K2} proved that for any given bounded 
sufficiently smooth domain $\Omega$, there exists $C_\Omega>0$ 
such that
\begin{equation}\label{eq:Greenestimate1}
|G_\Omega(x,y)|\leq C_\Omega|x-y|^{2k-n}\hbox{ for all }x,y\in \Omega,\; x\neq y.
\end{equation}
The constant $C_\Omega$ depends on $C^{2k,\theta}$-properties of the boundary $\partial \Omega$.
In Krasovski{\u\i}'s works, very general operators and boundary conditions
were discussed. Applying these general results to our special polyharmonic
Dirichlet problems originally required a higher degree of smoothness. However, it turns
out that for our purposes, $C^{2k,\theta}$-smoothness of $\partial \Omega$ suffices.
For more detailed information on this issue we refer to  Theorem \ref{Th:bnd} in the appendix. 
Estimate~(\ref{eq:Greenestimate1})  can also  be extended to the derivatives of
Green functions: For any $0\leq r\leq 2k$, there exists $C_{\Omega,r}$ such that
\begin{equation}\label{eq:Greenestimate2}
|\nabla_y^r G_\Omega(x,y)|\leq C_{\Omega,r}|x-y|^{2k-n-r}\hbox{ for all }x,y\in \Omega,\; x\neq y.
\end{equation}
Here, $\nabla_y^r $ denotes any partial derivative with respect to $y$ of order $r$.

\medskip\noindent 
The constant $C_\Omega$ in the Green function estimate (\ref{eq:Greenestimate1}) 
depends --as soon as $k>1$-- heavily on the smoothness properties of $\partial \Omega$.
As long as one considers families of domains with uniform smoothness properties 
one may choose the same constant. 

\smallskip\noindent In the present article, we exhibit families of domains with unbounded curvature, namely
fixed domains $\Omega$ where we punch out arbitrarily small holes. For uniform
Green function estimates, (\ref{eq:Greenestimate1}) can no longer be used since the curvature blows-up, and so does the constant $C(\Omega)$. Nevertheless, we can prove the following uniform estimates.

\begin{theorem}\label{Th:1} Let $\Omega$ be a $C^{2k,\theta}$-smooth bounded domain of $\rn$ and let $x_0\in \Omega$. 
Let $\omega$ be a $C^{2k,\theta}$-smooth bounded domain of $\rn$ containing $0$. We fix a number $q\in (0,1)$.
Then there exists a constant $C=C(\Omega,\omega,x_0,q)>0$ such that  that for all 
$\varepsilon\in (0,q\frac{d(x_0, \partial\Omega)}{\operatorname{diam}(\omega)})$, 
we have that
$$
|G_{\Omega_\varepsilon}(x,y)|\leq C|x-y|^{2k-n}\hbox{ for all }x,y\in \Omega_\varepsilon,\; x\neq y,
$$
where $\Omega_\varepsilon:=\Omega\setminus\overline{\{x_0+\varepsilon\omega\}}$. 
\end{theorem}

\begin{remark}
{\rm
In small dimensions $n\le 2k$, a uniform estimate like (\ref{bnd:g:eps}) below is no longer available
in the complements of arbitrarily small domains for $i=0$. Nakai and Sario \cite{NakaiSario} 
discussed the biharmonic case $k=2$ in dimension
$n=2$ with the help of energy estimates and their approach can probably be used
for any $k\ge 2$ and any dimension $n < 2k$. In this small dimensions case some (in general not all)
of the Dirichlet boundary conditions remain in $x_0$ even in the singular limit
$\Omega_0=\Omega\setminus \{ x_0\}$. This phenomenon cannot be expected in
large dimensions $n\ge 2k$. 
}
\end{remark}

\medskip\noindent 
It is then natural to ask whether in estimates like (\ref{eq:Greenestimate2}) we may also  expect
uniformity with respect to the family of domains $(\Omega_\varepsilon)_{\varepsilon}$.
This, however, is not the case.  More precisely, we have the following:
\begin{proposition}\label{prop:der} Let $\Omega$, $q\in(0,1)$, $\Omega_\eps$, $\eps>0$, be as in Theorem \ref{Th:1}. Then for all $1\leq r\leq 2k$, we have that
$$
\sup_{\eps\in (0,q d(x_0, \partial\Omega)/ \operatorname{diam}(\omega))}\sup_{x,y\in\Omega_\eps,\; x\neq y}|x-y|^{n-2k+r}|\nabla^r_y G_{\Omega_\eps}(x,y)|=+\infty.
$$
\end{proposition}

\medskip\noindent 
As mentioned at the beginning, one has a comparison principle for (\ref{eq:polyharmonic_bvp}) in general only in the second order case,
i.e. if $k=1$. In this case, $G_\Omega>0$ holds true for any $\Omega$, while if $k\ge 2$ one has positivity $G_\Omega>0$
only in very restricted classes of domains among which are balls (Boggio \cite{Boggio}) and small perturbations of balls (Grunau-Robert \cite{GR}). In general, however,
one has sign change, i.e. $G_\Omega\not\ge0$. Already Hadamard \cite{Hadamard} observed that this will occur in
the biharmonic case in two-dimensional annuli with very small inner radii, see also Nakai-Sario \cite{NakaiSario}.
On the other hand, for fixed domains, the negative part will be
``relatively'' small. For more detailed information on this issue one may see Grunau-Robert \cite{GR}, Gazzola-Grunau-Robert \cite{GGS} and Grunau-Robert-Sweers \cite{GRS}. For instance,
the authors proved in \cite{GR} that for any $C^{4,\theta}$-smooth bounded domain $\Omega\subset \rn$, $n>4$, 
there exists $C(\Omega)>0$ such that $\Vert(G_\Omega)_-\Vert_{L^\infty(\Omega)}\leq C(\Omega)$, where $G_\Omega$ is the 
Green function for $(-\Delta)^2$ with Dirichlet boundary condition. A natural question is to ask whether one may expect
uniformity of this lower  bound with respect to families of domains. 
As shown by the following proposition, the validity of this guess is equivalent to the nonnegativity of \emph{all} Green functions:

\begin{proposition}\label{prop:unbounded_negative_part}
We assume that $n>2k$. The two following assertions are equivalent:
\begin{itemize}
\item[(i)] There exists $C(k,n,\theta)$ depending only on $k,n,\theta$ such that such that 
$$\Vert (G_\Omega)_-\Vert_{L^\infty(\Omega)}\leq C(k,n,\theta)$$ 
for all $C^{2k,\theta}$-smooth bounded domains $\Omega\subset\rn$.
\item[(ii)] $G_\Omega\geq 0$ for all $C^{2k,\theta}$-smooth bounded domains $\Omega\subset\rn$.
\end{itemize}
\end{proposition}
Since (ii) is false for the higher order case $k\ge 2$ (see the discussion and references in the monograph 
Gazzola-Grunau-Sweers \cite[pp. 62/63 and 69/70]{GGS}) we conclude that  there is no uniform bound for negative parts of 
biharmonic and polyharmonic  Green functions. We emphasise that we only discuss Dirichlet boundary 
conditions and that positivity issues may be quite different for other boundary conditions.

\medskip\noindent{\it Notation:} In the sequel, $C(a,b,\ldots)$ denotes a constant depending on $\omega, \Omega, a,b,\ldots$. 
The same notation can be used for two different constants from line to line, and even in the same line.

\medskip\noindent{\bf Proofs.} 
 
We start with proving Theorem~\ref{Th:1} and proceed in several steps. In order to keep the exposition as simple as possible we shall prove the theorem for $q=\frac{1}{42}$. At the end of Step 3 we shall indicate how to modify the proof for larger $q<1$. Without loss of generality, we assume that $x_0=0$ so that $\Omega_\eps:=\Omega\setminus\eps\omega$. 

\medskip\noindent{\bf Step 1. The Green function in the exterior domain $\rn\setminus\omega$.} 

\medskip\noindent
Let $\omega$ be a $C^{2k,\theta}$ domain of $\rn$ such that $0\in\omega$. We define 
$$
\omega_0:=\inv(\rn\setminus\overline{\omega})\cup\{0\}\; , \, \hbox{where }\inv: \left\{\begin{array}{ccc}
  \rn\setminus \{0\} & \to & \rn\setminus \{0\},\\
  x & \mapsto &\frac{x}{|x|^2}.
\end{array}\right.
$$
We emphasise that $\inv$ is a special M\"obius transform of $\mathbb{R}^n$ and in particular
conformal.
The set $\omega_0$ is a $C^{2k,\theta}$-smooth bounded domain of $\rn$ containing $0$. We define 
\begin{equation}\label{eq:Greenfunction_small:hole}
G_{(\eps\omega)^c}(x,y):=\eps^{n-2k}|y|^{2k-n}|x|^{2k-n}G_{\omega_0}(\eps\inv(x), \eps\inv(y))
\end{equation}
for all $x,y\in \rn\setminus\varepsilon\omega$.  The following proposition shows that this is indeed the polyharmonic
Green function in $(\eps\omega)^c$:

\begin{proposition}
For any $\varphi\in C^{2k}_c(\rn\setminus \eps\omega)$ such that $\partial_\nu^{(i)}\varphi=0$ on $\partial (\eps\omega)$
for $i=0,\ldots,(k-1)$, we have that
\begin{equation}\label{eq:green:ext}
\varphi(x)=\int_{\rn\setminus \eps\omega}G_{(\eps\omega)^c}(x,y)(-\Delta)^k\varphi(y)\, dy
\end{equation}
for all $x\in \rn\setminus \eps\omega$. Moreover, for all $0\leq i\leq 2k$, the derivatives with respect to $y$ satisfy the upper bound
\begin{equation}\label{bnd:g:eps}
|\nabla^i_y G_{(\varepsilon \omega)^c}(x,y)|\leq C|y|^{-i}\sum_{r\leq i}|x|^r|x-y|^{2k-n-r}.
\end{equation}
\end{proposition}

\begin{proof} We prove the claim first for $\varepsilon=1$.
Let $\varphi\in C^{2k}_c(\rn\setminus \omega)$ be such that $\partial_\nu^{(i)}\varphi=0$ on $\partial \omega$
for $i=0,\ldots,(k-1)$. We show that
\begin{equation}\label{claim:1}
\varphi(x)=\int_{\rn\setminus \omega}|y|^{2k-n}|x|^{2k-n}G_{\omega_0}(\inv(x), \inv(y))(-\Delta)^k\varphi(y)\, dy
\end{equation}
for all $x\in \rn\setminus \omega$.

\smallskip\noindent   Indeed, $\inv$ is the composition of two sterographic projections of 
opposite poles, and therefore, it is conformal and the pull-back of the Euclidean metric $\Eucl$ via $\inv$ is 
$\hbox{inv}^\star \Eucl=|\, \cdot\, |^{-4}\Eucl=\mu^{4/(n-2k)}\Eucl$ where $\mu(x):=|x|^{2k-n}$ for all 
$x\in\rn\setminus\{0\}$. As a consequence, considering $(-\Delta)^k$ as the conformal operator of 
Graham-Jenne-Mason-Sparling for the Euclidean space (see \cite{GJMS}), the conformal law of the 
GJMS operators yields
$$((-\Delta)^k\varphi)\circ \inv=\mu^{-(n+2k)/(n-2k)}(-\Delta)^k(\mu(\varphi\circ\inv)).$$
In addition, the Jacobian of $\inv$ and then the Riemannian element of volume of $\hbox{inv}^\star \Eucl$ are
$$\Jac(\inv)=|\, \cdot \, |^{-2n}\hbox{ and }dv_{\inv^\star\Eucl}=|\, \cdot\, |^{-2n}\, dx.$$
This transformation behaviour of polyharmonic operators
with respect to M\"obius transforms is classical, see e.g. Loewner \cite{Loewner} and references therein.
A convenient and easily accessible reference is also Gazzola-Grunau-Sweers \cite[Lemma 6.14]{GGS} .

\medskip\noindent We fix $x\in \rn\setminus \overline{\omega}$ and we consider
$x':=\inv(x)\in \omega_0\setminus\{0\}$. We define $\tilde{\varphi}(y):=\mu(y)\varphi\circ\inv(y)=|y|^{2k-n}\varphi(y/|y|^2)$ 
for $y\in \overline{\omega}_0\setminus\{0\}$. We find that $\tilde{\varphi}$ is vanishing around $0$ and therefore 
extends smoothly to $\overline{\omega}_0$. It follows from Green's representation formula that
$$\tilde{\varphi}(x')=\int_{\omega_0}G_{\omega_0}(x', y)(-\Delta)^k \tilde{\varphi}(y)\, dy.$$
Performing the change of variable $y=\inv(z)$ and using the above properties yields
$$
\tilde{\varphi}(x')=\int_{\rn\setminus \omega}|z|^{n+2k}G_{\omega_0}(x', \inv(z))(-\Delta)^k\varphi(z)|z|^{-2n}\, dz.
$$
Going back to the expression of $\varphi$ yields \eqref{claim:1}.

\medskip\noindent Given $\alpha$ a multi-index and $j\in\{1,\ldots,n\}$, there exists an homogeneous polynomial 
$P^\alpha_j$ of degree $|\alpha|+1$ such that
$$\partial^\alpha\inv(x)_j=\frac{P^\alpha_j(x)}{|x|^{2(|\alpha|+1)}}$$
for all $x\in\rn\setminus\{0\}$, where $|\alpha|$ is the length of the index. We fix $x,y\in \rn\setminus\omega$ such that 
$x\neq y$.   With help of the binomial formula, the derivative of order $\alpha$ with respect to $y$ 
is such that
\begin{equation}\label{ineq:1}
| \partial^\alpha_y G_{\omega^c}(x,y)|
\leq C|x|^{2k-n}\sum_{\beta\leq\alpha}|y|^{2k-n-|\alpha|+|\beta|}|\partial_y^\beta (G_{\omega_0}(\inv(x),\inv(y)))|,
\end{equation}
where we have adopted the standard order on multi-indices. For $|\beta|\geq 1$, the chain rule yields 
$$\partial^\beta(f\circ \inv)=\sum_{1\leq r\leq |\beta|}\sum_{I_1 + \ldots + I_r
=\beta}\sum_{j_1,\ldots,j_r}c^{(I_1,\ldots, I_r)}_{j_1,\ldots , j_r}
\partial^{I_1}\inv_{j_1}\ldots\partial^{I_r}\inv_{j_r}(\partial^{j_1\ldots  j_r}f)\circ \inv$$
for any function $f$ when the derivatives make sense. The second sum is taken over all decompositions of $\beta$ 
as a sum of  $r$ multi-indices and the $c^{(I_1,\ldots, I_r)}_{j_1,\ldots , j_r}$ are combinatorial
constants which can  be calculated explicitly. When restricting to suitable decompositions of $\beta$ these constants are equal to $1$. This  formula yields
\begin{equation}\label{ineq:2}
|\partial^\beta_y (G_{\omega_0}(\inv(x),\inv(y)))|\leq C\sum_{r\leq |\beta|}|y|^{-|\beta|-r}|(\nabla^r G_{\omega_0})(\inv(x),\inv(y))|
\end{equation}
for all $\beta\leq \alpha$. Here, $\nabla^r f=(\partial^\gamma f)_{|\gamma|=r}$ when this makes sense.

\medskip\noindent It follows from Krasovski{\u\i} \cites{K1,K2} that  for any $0\leq r\leq 2k$, there exists 
$C=C(\omega_0=\inv(\overline{\omega}^c)\cup\{0\},r)>0$ such that 
$$|\nabla^rG_{\omega_0}(x,y)|\leq C|x-y|^{2k-n-r}$$
for all $x,y\in\omega_0$, $x\neq y$. For the sake of completeness, we refer to Theorem \ref{Th:bnd} in the appendix 
where we comment on an alternative to Krasovski{\u\i}'s proof. Noting that
\begin{equation}\label{ineq:3}
|\inv(x)-\inv(y)|=\frac{|x-y|}{|x|\cdot|y|}
\end{equation}
and putting \eqref{ineq:1}, \eqref{ineq:2} and \eqref{ineq:3}  together  yields
$$
|\partial_y^\alpha G_{\omega^c}(x,y)|\leq C|y|^{-|\alpha|}\sum_{r\leq |\alpha|}|x|^r|x-y|^{2k-n-r}.
$$
This proves the claim for $\eps=1$, while for arbitrary $\eps>0$ it follows from the previous reasoning and the observation
that  $G_{(\eps\omega)^c}(x,y):=\eps^{2k-n}G_{\omega^c}(x/\eps,y/\eps)$.
\end{proof}

\medskip\noindent{\bf Step 2. Control outside a small annulus.}   

\medskip\noindent
Given $\delta\in \left(0,\frac{d(0,\partial\Omega)}{3}\right)$, we define $\eta_\delta\in C^\infty_c(\Omega)$ 
such that $\eta_\delta(x)=1$ for all $x\in B_\delta(0)$ and $\eta_\delta(x)=0$ for all $x\in \Omega\setminus B_{2\delta}(0)$. 
Given $\eps\in (0,\frac{\delta }{2\operatorname{diam}(\omega)})$ and $x,y\in \Omega_\eps$, we define
\begin{equation}\label{def:tG}
\tilde{G}_{\eps,\delta}(x,y):=\eta_\delta(y)G_{(\eps\omega)^c}(x,y)+(1-\eta_\delta(y))G_{\Omega}(x,y).
\end{equation}
We get that
\begin{multline*}
(-\Delta)^k \tilde{G}_{\eps,\delta}(x,\, \cdot\, )=\eta_\delta(-\Delta)^kG_{(\eps\omega)^c}(x,\, \cdot\, )
+(1-\eta_\delta)(-\Delta)^kG_{\Omega}(x,\, \cdot\, )\\+\sum_{i<2k}\left(A_i(\nabla^{2k-i}\eta_\delta,\nabla^iG_{(\eps\omega)^c}(x,\, \cdot\, ))
+A_i(\nabla^{2k-i}(1-\eta_\delta),\nabla^iG_{\Omega}(x,\, \cdot\, ))\right),
\end{multline*}
where the $A_i's$ are contractions of suitable tensors, that is bilinear forms with smooth coefficients. 
Therefore, for any $x\in \Omega_\eps$, there exists $f_{\eps,\delta,x}$ such that
$$
(-\Delta)^k \tilde{G}_{\eps,\delta}(x,\, \cdot\, )=\delta_x+f_{\eps,\delta,x}\hbox{ in }{\mathcal D}'(\Omega_\eps).
$$
Moreover, the pointwise control \eqref{bnd:g:eps} yields
$$
|f_{\eps,\delta,x}(y)|\leq C \, \cdot\, {\bf 1}_{B_{2\delta}(0)\setminus B_{\delta}(0)}|x-y|^{1-n}
$$
for all $x,y\in \Omega_\eps$, $x\neq y$. In particular, there exists $C(\delta)>0$ such that
\begin{equation}\label{bnd:f:eps}
\Vert f_{\eps,\delta,x}\Vert_{L^\infty(\Omega_\eps)}\leq C(\delta)\hbox{ for all }\eps>0\hbox{ and }x\in \Omega_{\eps,\delta}
\end{equation}
where
$$
\Omega_{\eps,\delta}:=\left(\Omega_\eps\cap B_{\delta/2}(0)\right)\cup \left(\Omega_\eps\setminus \overline{B_{3\delta}(0)}\right)
=\Omega_\eps \setminus \left(\overline{B_{3\delta}(0)} \setminus B_{\delta/2}(0)\right).
$$
Then it follows from elliptic theory that for any $x\in \Omega_{\eps,\delta}$, there exists $u_{x,\eps,\delta}\in W^{k,2}_0(\Omega_\eps)$ such that
\begin{equation}
\left\{\begin{array}{ll}\label{eq:uxd}
(-\Delta)^k u_{x,\eps,\delta}=f_{\eps,\delta,x} & \hbox{ in }\Omega_\eps,\\
\partial_\nu^{(i)} u_{x,\eps,\delta}=0& \hbox{ for all }i=0,\ldots,k-1\hbox{ on }\partial\Omega_\eps.
\end{array}\right.
\end{equation}

\medskip\noindent 
We claim that $u_{x,\eps,\delta}\in C^{2k-1}(\Omega_\eps)$ for all $\eps,\delta>0$ and $x\in \Omega_{\eps,\delta}$. 
Moreover, there exists $C(\delta)>0$ such that
\begin{equation}\label{bnd:u:eps}
\Vert u_{x,\eps,\delta}\Vert_{C^{2k-1}(\Omega\setminus B_{\delta/4}(0))}\leq C(\delta)
\end{equation}
for all admissible $\eps,\delta>0$ and $x\in \Omega_{\eps,\delta}$.

\medskip\noindent 
We prove this claim. For simplicity, we define 
$$((-\Delta)^{k/2}\psi)^2:=\left\{\begin{array}{ll}
((-\Delta)^{l}\psi)^2 & \hbox{ if }k=2l\hbox{ is }even\\
|\nabla (-\Delta)^l\psi|^2& \hbox{ if }k=2l+1\hbox{ is }odd.
\end{array}\right.$$
As a consequence, $u\mapsto \Vert (-\Delta)^{k/2}u\Vert_2$ is a norm on $W^{k,2}_0(\Omega_\eps)$, the completion of $C^\infty_c(\Omega_\eps)$ for the usual norm. 
Multiplying \eqref{eq:uxd} by $u_{x,\eps,\delta}$ and integrating by parts yields with H\"older's inequality
\begin{eqnarray*}
\int_{\Omega}((-\Delta)^{k/2} u_{x,\eps,\delta})^2\, dx&=&\int_{\Omega_\eps}((-\Delta)^{k/2} u_{x,\eps,\delta})^2\, dx\\
&=&\int_{\Omega_{\eps}}f_{\eps,\delta,x}u_{x,\eps,\delta}\, dy\leq \Vert f_{\eps,\delta,x}\Vert_{\frac{2n}{n+2k}}\Vert u_{x,\eps,\delta}\Vert_{\frac{2n}{n-2k}}.
\end{eqnarray*}
Sobolev's inequality yields the existence of $C_{n,k}>0$ such that 
$$\Vert u\Vert_{\frac{2n}{n-2k}}\leq C_{n,k}\Vert (-\Delta)^{k/2} u\Vert_2$$
for all $u\in C^\infty_c(\rn)$. The density of $C^\infty_c(\Omega_\eps)$ in $W^{k,2}_0(\Omega_\eps)$ 
allows to conclude that
$$
\Vert u_{x,\eps,\delta}\Vert_{\frac{2n}{n-2k}}^2\leq C^2_{n,k} \Vert f_{\eps,\delta,x}\Vert_{\frac{2n}{n+2k}}\Vert u_{x,\eps,\delta}\Vert_{\frac{2n}{n-2k}}
$$
for all $\eps>0$ and $x\in \Omega_{\eps,\delta}$. Therefore $\Vert u_{x,\eps,\delta}\Vert_{\frac{2n}{n-2k}}\leq C'(\delta)$.

\medskip\noindent It follows from elliptic theory (see for instance Agmon-Douglis-Nirenberg \cite{ADN}) that for all 
$p>1$ and all $\delta'>0$, there exists $C(\delta')>0$ such that
$$
\Vert u_{x,\eps,\delta}\Vert_{W^{2k,p}(\Omega\setminus B_{\delta'}(0))}
\leq C(\delta',p,\Omega)(\Vert f_{\eps,\delta,x}\Vert_p+\Vert u_{x,\eps,\delta}\Vert_p).
$$
The claim   \eqref{bnd:u:eps} follows from this inequality, Sobolev's inequalities and iterations.

\medskip\noindent 
It remains to gain control of $u_{x,\eps,\delta}$ in $B_{\delta'}(0) \setminus (\varepsilon\omega)$. To this
end we consider $\eta_\delta u_{x,\eps,\delta}$ and observe that this function solves a Dirichlet problem
in the exterior domain $(\overline{\varepsilon\omega})^c$. Indeed, we have that
\begin{eqnarray*}
(-\Delta)^k(\eta_\delta u_{x,\eps,\delta})&=& \eta_\delta(-\Delta)^k u_{x,\eps,\delta}+\sum_{i<2k}A_i(\nabla^{2k-i}\eta_\delta,\nabla^i u_{x,\eps,\delta})\\
&=&  \eta_\delta f_{\eps,\delta,x}+\sum_{i<2k}A_i(\nabla^{2k-i}\eta_\delta,\nabla^i u_{x,\eps,\delta}):=\tilde{f}_{\eps,\delta,x}
\end{eqnarray*}
where the $A_i$'s are as above.  We observe that
$$
\operatorname{supp }\tilde{f}_{\eps,\delta,x}\subset B_{2\delta}(0)\hbox{ and }\Vert \tilde{f}_{\eps,\delta,x}\Vert_\infty\leq C(\delta)
$$
for all $x\in \Omega_{\eps,\delta}$. Since $\eta_\delta u_{x,\eps,\delta}$ has compact support in $\rn\setminus\eps\omega$ and vanishes up to $(k-1)^{th}$ order on 
$\partial(\eps\omega)$, Green's representation formula \eqref{eq:green:ext} yields
$$(\eta_\delta u_{x,\eps,\delta})(z)=\int_{\rn\setminus \eps\omega} G_{(\eps\omega)^c}(z,y)\tilde{f}_{\eps,\delta,x}(y)\, dy$$
for all $z\in \rn\setminus \overline{\eps\omega}$. Consequently, for any $z\in B_\delta(0)\setminus (\eps\omega)$, one gets
\begin{eqnarray*}
| u_{x,\eps,\delta}(z)|&=&|(\eta_\delta u_{x,\eps,\delta})(z)|\leq \int_{\rn\setminus \eps\omega} |G_{(\eps\omega)^c}(z,y)\tilde{f}_{\eps,\delta,x}(y)|\, dy\\
&\leq& C(\delta) \int_{B_{2\delta(0)}\setminus \eps\omega} |y-z|^{2k-n}\, dz\leq C(\delta).
\end{eqnarray*}
This inequality combined with \eqref{bnd:u:eps} yields
\begin{equation}\label{bnd:u}
\Vert u_{x,\eps,\delta}\Vert_{L^\infty(\Omega_\eps)}\leq C(\delta)\hbox{ for all }x\in \Omega_{\eps,\delta}.
\end{equation}
As a consequence, we find that 
$$(-\Delta)^k(\tilde{G}_{\eps,\delta}(x,\, \cdot\, )-u_{x,\eps,\delta})=\delta_x\hbox{ weakly in }{\mathcal D}'(\Omega_\eps)$$
and $\partial_\nu^{(i)}(\tilde{G}_{\eps,\delta}(x,\, \cdot\, )-u_{x,\eps,\delta})=0$ on $\partial\Omega_\eps$ for all $x\in \Omega_{\eps,\delta}$ and all $i=0,\ldots,k-1$. 
The uniqueness of the Green function implies that
\begin{equation}\label{eq:G:tG}
G_{\Omega_\eps}(x,\, \cdot\, )=\tilde{G}_{\eps,\delta}(x,\, \cdot\, )-u_{x,\eps,\delta}
\end{equation}
and then, using \eqref{eq:green:ext} and \eqref{bnd:u}, we arrive at 
\begin{eqnarray}\label{bnd:G:1}
|G_{\Omega_\eps}(x,y)|&\leq &|\tilde{G}_{\eps,\delta}(x,y)|+|u_{x,\eps,\delta}(y)|\leq C(\Omega,\omega)|x-y|^{2k-n}+C(\delta)\\
&\leq &C(\delta)|x-y|^{2k-n}
\end{eqnarray}
for all $x\in \Omega_{\eps,\delta}=(\Omega_\eps\cap B_{\delta/2})\cup (\Omega_\eps\setminus \overline{B}_{3\delta}(0))$ 
and all $y\in\Omega_\eps$. 

\medskip\noindent{\bf Step 3.}

\begin{proof}[Conclusion of the proof of Theorem~\ref{Th:1}]
We fix $\delta_0\in (0,\frac{d(x_0,\partial \Omega)}{21})$. We apply Step 2 with $\delta:=\delta_0$ and to 
$\delta:=7\delta_0$. Since $\Omega_{\varepsilon}=\Omega_{\varepsilon,\delta_0}\cup \Omega_{\varepsilon,3\delta_0}$, it  follows from \eqref{bnd:G:1} 
that there exists $C>0$ such that
$$|G_{\Omega\setminus\overline{\eps\omega}}(x,y)|\leq C|x-y|^{2k-n}\hbox{ for all }x,y\in \Omega\setminus\overline{\eps\omega}, \; x\neq y.$$
This proves Theorem \ref{Th:1} for $q=1/42$. For $q\in (1/42,1)$, instead of $\delta/2$, $\delta$, $2 \delta$, $3\delta $, in Step 2 one has to work with
$\delta/(1+\sigma)$, $\delta$, $(1+\sigma)\delta$, $(1+2\sigma)\delta$ with $\sigma >0$ sufficiently close to $0$. Alternatively
one may argue that for 
$$\varepsilon \in [ (1/42) d(x_0, \partial\Omega)/ \operatorname{diam}(\omega),
 q d(x_0, \partial\Omega)/ \operatorname{diam}(\omega))
$$ 
the boundaries of the $\Omega_\varepsilon$
enjoy uniform $C^{2k,\theta}$-properties so that (\ref{eq:Greenestimate1}) holds uniformly with
respect to these $\varepsilon $.
\end{proof}

\newpage\noindent {\bf Step 4.} 
\begin{proof}[Proof of Proposition~\ref{prop:der}]
We argue by contradiction and assume that there exist $1\leq r\leq 2k$ and $C>0$ such that 
\begin{equation}\label{hyp:bnd:grad}
|x-y|^{n-2k+r}|\nabla^r_y G_{\Omega_\eps}(x,y)|\leq C
\end{equation}
for all $x,y\in \Omega_\eps$, $x\neq y$, uniformly in $\varepsilon\to 0$. 
For any $x,y\in (\eps^{-1}\Omega)\setminus \omega$, we define $G_\eps(x,y):=\eps^{n-2k}G_{\Omega_\eps}(\eps x,\eps y)$. 
It follows from \eqref{def:tG}, \eqref{bnd:u:eps}, and \eqref{eq:G:tG} that for any $x\in \omega^c$, we have that
$$\lim_{\eps\to 0}G_\eps(x,y)=G_{\omega^c}(x,y)$$
in $C^{0}_{loc}(\rn\setminus(\omega\cup\{x\}))$. Since $(-\Delta)^k G_\eps(x,\, \cdot\, )=0$ and $G_\eps(x,\, \cdot\, )$ vanishes on $\partial\omega$ up to order $(k-1)$, 
elliptic regularity yields convergence in $C^{2k}_{loc}(\omega^c\setminus \{x\})$. Rewriting \eqref{hyp:bnd:grad} for $G_\eps$ and passing to the limit $\eps\to 0$ yields
\begin{equation}\label{hyp:bnd:grad:2}
|x-y|^{n-2k+r}|\nabla^r_y G_{\omega^c}(x,y)|\leq C
\end{equation}
for all $x,y\in \rn\setminus\omega$, $x\neq y$. We fix $x\neq 0$ and we define $G_R(z):=R^{n-2k}G_{\omega^c}(Rx,z)$ 
for all $z\in \omega^c$ and $R>R_0$ large enough. 
It follows from the explicit expression of $G_{\omega^c}$ in (\ref{eq:Greenfunction_small:hole}) that
\begin{equation}\label{cv:G}
\lim_{R\to +\infty}G_R(z)=G(z):=|x|^{2k-n}|z|^{2k-n}G_{\omega_0}\left(0,\frac{z}{|z|^2}\right)
\end{equation}
in $C^0_{loc}(\rn\setminus\omega)$. Since $(-\Delta)^kG_R=0$ and $G_R$ vanishes on $\partial\omega$ up to order $(k-1)$, elliptic regularity 
yields the convergence of $(G_R)$ 
to $G$ in $C^{2k}_{loc} (\omega^c)$. On the other hand, \eqref{hyp:bnd:grad:2} may be rewritten as
$$
|\nabla^r G_R(z)|\leq C R^{-r}|x-R^{-1}z|^{2k-n-r}
$$
for $z$ in a compact sudomain of $\rn\setminus\omega$ and $R$ large enough. Since $r\geq 1$, 
passing to the limit $R\to +\infty$ yields $\nabla^r G=0$ in $\rn\setminus\omega$, 
which contradicts the explicit expression \eqref{cv:G} of $G$. This concludes the proof of Proposition~\ref{prop:der}.
\end{proof}


\medskip\noindent {\bf Step 5.}
\begin{proof}[Proof of Proposition~\ref{prop:unbounded_negative_part}] 
Assume that (ii) does not hold. Then there exists a $C^{2k,\theta}$-smooth bouned domain $\omega_0\subset\rn$ such that $G_{\omega_0}$ attains some negative values, 
say at $(x_0,y_0)\in\omega_0\times\omega_0$, $x\neq 0$. We define $\omega:=(\inv(\rn\setminus \overline{\omega_0}))\cup\{0\}$ and 
$\Omega_\eps:=\Omega\setminus\overline{\eps\omega}$ where $\eps>0$ is small and $\Omega$ is a smooth bounded domain containing $0$. 
It follows from \eqref{def:tG}, \eqref{bnd:u}, and \eqref{eq:G:tG} that 
$$
\lim_{\eps\to 0}\eps^{n-2k}G_{\Omega_\eps}(\eps x,\eps y)=G_{\omega^c}(x,y)=|x|^{2k-n}|y|^{2k-n}G_{\omega_0}(\inv(x), \inv(y))
$$
for all $x,y\in \rn\setminus\overline{\omega}$. Choosing $x:=\inv(x_0)$ and $y:=\inv(y_0)$ yields 
$$\lim_{\eps\to 0}G_{\Omega_\eps}(\eps x,\eps y)=-\infty,$$
and then (i) does not hold. Conversely, if (ii) holds, then (i) holds.
\end{proof}

\appendix
\section{Pointwise control of the Green function for fixed domains}

The following result, under stronger smoothness assumptions on $\Omega$  but at the same time  
in a more general  context,  is due to Krasovski{\u\i} \cites{K1,K2}:
\begin{theorem} \label{Th:bnd}Let $\Omega\subset\rn$ be a $C^{2k,\theta}$-smooth bounded domain of $\rn$ with $2k<n$,
$\theta \in (0,1)$, and $k\geq 1$. 
Let $G_\Omega$ be the Green function for $(-\Delta)^k$ 
with Dirichlet boundary condition. Then for all $0\leq r\leq 2k$, there exists $C(\Omega, r)>0$ such that
\begin{equation}\label{ineq:nabla:G}
|\nabla^r_yG_{\Omega}(x,y)|\leq C|x-y|^{2k-n-r}
\end{equation}
for all $x,y\in\Omega$, $x\neq y$.
\end{theorem}
We sketch here an alternative proof.
\begin{proof} The case $r=0$ and $k=2$ under the smoothness assumptions as in the theorem 
is treated in  Grunau-Robert \cite[Theorem 4]{GR} (see also Gazzola-Grunau-Sweers \cite[Propositions 4.22 and 4.23]{GGS} for an exposition in 
book form). By making the obvious changes one may check that the proof can be extended to any $k\geq 1$
and $n>2k$. (Only the discussion of the smaller dimensions $n\le 2k$ requires more care.)
This means that there exists a constant $C(\Omega)>0$ such that 
\begin{equation}\label{bnd:G:0}
|G_\Omega(x,y)|\leq C(\Omega)|x-y|^{2k-n}
\end{equation}
for all $x,y\in\Omega$, $x\neq y$. We fix $r\geq 1$ and we prove \eqref{ineq:nabla:G} by using
local elliptic estimates and rescaling arguments. We proceed as in Grunau-Gazzola-Sweers \cite[Prop. 4.23]{GGS}
and use the following local Schauder estimate from Agmon-Douglis-Nirenberg \cite[Theorem 9.3]{ADN}
which holds true also close to $\partial \Omega$.
 For any two
concentric balls $B_{R}\subset B_{2R}$ 
and any polyharmonic function $v$ on  $B_{2R}\cap \Omega$
satisfying homogeneous Dirichlet boundary conditions on $B_{2R}\cap \partial\Omega$
we have
\begin{equation} \label{scaledlocalboundaryestimate.2}
\Vert \nabla^r v\Vert _{L^{\infty }(B_{R}\cap \Omega)}
\leq \frac{C}{R^r}\Vert
v\Vert _{L^{\infty }(B_{2R}\cap \Omega )}. 
\end{equation}
The constant is uniform in $R$; the behaviour with respect to (small)
$R$ is obtained by means of scaling.

Keeping $x\in \Omega$ fixed, for
any $y\in \Omega \setminus \{x\}$ we choose $R=|x-y|/4$ and apply 
(\ref{scaledlocalboundaryestimate.2}) and (\ref{bnd:G:0})  in 
$B_{R}(y)\subset B_{2R}(y)$ to $G_{\Omega}(x,\, \cdot \, )$. 
This proves (\ref{ineq:nabla:G}).

\end{proof}

\end{document}